\documentclass[journal]{IEEEtran}

\usepackage[english]{babel}
\usepackage{amsmath,amssymb,amsthm}
\usepackage{graphicx}
\usepackage{booktabs,tabularx,array}
\usepackage{enumitem}
\usepackage{siunitx}
\usepackage[hidelinks]{hyperref}
\usepackage{tikz}
\usetikzlibrary{arrows.meta,positioning}

\allowdisplaybreaks

\newtheorem{theorem}{Theorem}

\newcommand{\FA}{\ensuremath{\mathrm{FA}}}
\newcommand{\TD}{\ensuremath{\mathrm{TD}}}
\newcommand{\TDTR}{\ensuremath{\mathrm{TDTR}}}

\newcommand{\NFK}{\ensuremath{\mathrm{NF}^{K}}}
\newcommand{\NFR}{\ensuremath{\mathrm{NF}^{R}}}
\newcommand{\NFC}{\ensuremath{\mathrm{NF}^{C}}}

\newcommand{\ApproveFixed}{\ensuremath{\mathrm{ApproveFixed}}}
\newcommand{\ApproveRepairable}{\ensuremath{\mathrm{ApproveRepairable}}}
\newcommand{\NoMaterialProduct}{\ensuremath{\mathrm{NoMaterialProduct}}}
\newcommand{\RejectReinforce}{\ensuremath{\mathrm{Unsupported}}}
\newcommand{\NoEvidence}{\ensuremath{\mathrm{NoEvidence}}}

\newcommand{\dimp}{\ensuremath{\mathrm{imp}}}
\newcommand{\dexp}{\ensuremath{\mathrm{exp}}}

\newcommand{\GOPACS}{\ensuremath{\mathrm{GOPACS}}}
\newcommand{\PTDF}{\ensuremath{\mathrm{PTDF}}}

\newcommand{\ind}[1]{\mathbf{1}\!\left\{#1\right\}}

\title{From Product Evidence to Access Rights:\\
Auditing and Routing Non-Firm Grid Connections}

\author{Alban Kryeziu, Stuart X. Zhu, Evrim Ursavas%
\thanks{A. Kryeziu, S. X. Zhu, and E. Ursavas are with the Faculty of
Economics and Business, Department of Operations, University of Groningen,
Groningen, The Netherlands
(e-mail: a.kryeziu@rug.nl; x.zhu@rug.nl; e.ursavas@rug.nl).}%
\thanks{This work has been submitted to the IEEE for possible publication.}}

\begin{document}
\maketitle

\begin{abstract}
Non-firm connection products can accelerate grid access while reinforcements are delayed, but their planning value depends on evidence supporting the contractual promise. We develop a one-way evidence-to-access framework for calendarized products. Stage A constructs and fixes directional restriction calendars from public residual-load and day-ahead price-timing signals. Stage B audits the submitted calendar and one bounded common repair, tests joint
threshold attainability under the shared annual budget, and applies an
aggregation-aware Power Transfer Distribution Factor (PTDF) certificate. Stage C treats the resulting ledger as fixed in queue-aware allocation and intervention planning. In the primary policy-package evaluation, the evidence-supported menu
returns the same $2.119$~GW allocation as firm-only access, whereas authorizing the uniformly
derated fallback raises awards to $6.442$~GW and reduces backlog from
$136.076$ to $113.601$~GW-years; all $4.489$~GW of non-firm awards use the
fallback. In a target-free matched control, evidence-routed and
state-blind fallback menus attain primary objectives equivalent within EUR~1
and return the same accepted applicants, award years, routes, and
intervention portfolio. Thus, in the reported testbed, fallback authorization rather than
calendar-specific evidence support drives the access-volume effect.
\end{abstract}


\section*{Nomenclature}
\vspace{-0.45em}
\begingroup
\footnotesize
\setlength{\itemsep}{0pt}
\setlength{\parskip}{0pt}
\begin{IEEEdescription}[
\IEEEusemathlabelsep
\IEEEsetlabelwidth{$H_{0,nyd},H_{R,nyd},H^{J,\star}_{ny}$}
]
\item[$\TDTR,\FA$]
Time-dependent transport right, firm access.

\item[$\GOPACS,\PTDF$]
Dutch congestion-management platform and power transfer distribution factor.
\item[$\NFK,\NFR,\NFC$]
Fixed-calendar, repaired-calendar, and uniformly derated fallback routes.

\item[$n,y,t,d;\ q,k,j$]
Province, planning year, physical hour, direction; season, hour of week,
and canonical cell $j=(q,k)$.

\item[$g,a,X,r,\sigma$]
Applicant, intervention, contractual route, access menu, and recourse
signature.

\item[$RL^0,S^-,S^+$]
Residual load, directional construction scores.

\item[$\alpha,K_y,\eta_p$]
Annual restriction share, corresponding hour budget, and price-timing weight.

\item[$\rho_0,\mu_{\min},O_{\min},B_0$]
Retained-value threshold, directional-materiality threshold, offered-direction
threshold, and repair edit budget.

\item[$H_0,H_R,H^{J,\star}$]
Submitted-calendar score, bounded-repair score, and joint-oracle score.
\item[$c_{nyd},X_g,\zeta^{rX}_{gy}$]
Directional evidence state, horizon-stable applicant route, and route gate.

\item[$H,\delta_C,s_C^{\mathrm{cost}}$]
Priority horizon, fallback deliverability, and fallback cost index.

\item[$\delta_{gX},w_{gy},a_{gyX};\ K^p_{ny}(x)$]
Route indicator, award-year indicator, awarded capacity, and available
capacity in pool $p$.

\item[$B^r,\operatorname{BRI}^r,J_{\mathrm{PP}}^r,J_{\mathrm{TF}}^r$]
Backlog exposure, backlog-relief index, policy-package objective, and
target-free objective.
\item[$\gamma_{\mathrm{rec}},\gamma_{\mathrm{val}},
\lambda_{\mathrm{conn}}$]
Operating-recourse, project-value, and terminal-connection weights.
\end{IEEEdescription}
\endgroup
\vspace{-0.55em}

\section{Introduction}
\label{sec:introduction}

Meeting the Paris Agreement targets \cite{paris} and advancing the energy
transition require rapid electrification and timely grid connection. Across
Europe and many power systems worldwide, major transmission reinforcements face
long permitting and construction lead times
\cite{iea_grids_secure_2023,entsoe_grids_package_response_2025}, while
connection demand and congestion must be managed on seasonal, day-ahead (DA),
and intraday time scales
\cite{iea_seasonal_variability_2024,acer_core_rdct_annexI_2020}. This
planning--operation mismatch has increased interest in conditional (non-firm)
grid-access arrangements that unlock interim hosting capacity through explicit
curtailment rules. When interim capacity is rationed through queues and
priority, however, both allocation and activation rules must be transparent,
data-driven, and auditable \emph{ex ante}, rather than ad hoc or
discretionary.

The Netherlands provides a timely setting. The Dutch
\emph{Netcode Elektriciteit} introduced the time-dependent transport right
(TDTR) on TenneT's grid, guaranteeing the complete transport right during at
least $85\%$ of annual time; a 2026 code decision refined its operational
conditions \cite{ACM_ATR_2024,ACM_TDTR_2026}. The product studied
here is not a reconstruction of the regulated TDTR, which is activated
operationally. It is a transparent calendarized planning abstraction with a
comparable annual-duration budget. For such a product, advance timing certainty
has value only if the recurring restriction windows are defensible
\emph{ex ante}. The relevant question is not whether public data can approve a
connection, but whether they support the submitted calendar, only a limited
repair of that calendar, or only an alternative fallback right.

Public data can reveal recurring patterns in demand, renewable production,
prices, and congestion-management activity
\cite{openpowersystemdata2024,zomerdijk_open_data}. They cannot reproduce an
operator connection study because spatial aggregation, confidentiality,
unobserved remedial actions, and incomplete network parameters leave material
physical information unresolved \cite{liu_delft_model_free}. We therefore use
provincial residual load as a directional physical-pressure proxy and the
national day-ahead price as a market-timing covariate, not as a locational
congestion price or deliverability certificate
\cite{Staudt2019CongestionPrediction,Titz2024CongestionDrivers,
Trebbien2023ElectricityPrices}. Candidate calendars are fixed before GOPACS, PTDF, queue, and intervention outcomes are evaluated. The goal is to grade the evidence supporting a contractual promise, not to infer operator-grade feasibility from public data.

Existing research addresses adjacent expansion, congestion-management, and
connection-queue decisions, but the evidentiary strength of a proposed
calendarized access product is typically not modeled as a preceding decision. Long-term grid
expansion is commonly formulated as a multi-period mixed-integer planning
problem with discrete investments, construction lead times, and operational
constraints \cite{ieee_2003_capacity_milp,LUMBRERAS201619,filippo_pecci}. These models
determine what to build, where to build it, and when new capacity becomes
available. Flexibility and
congestion-management studies value or activate redispatch, dynamic tariffs,
capacity-limitation contracts, local markets, and grid-enhancing technologies
\cite{SHENDTU,VANDERHOLST2024110664,
vanDerHolst2026TEMPR,Spyrou2025TEMPR,Rui2024TEMPR,9312120}. Queue-management research examines the readiness, sequencing, withdrawal risk, connection timing and flexible agreements \cite{electricity_grid_expansion_interconnection_queue,managing_queues_flexible_contract,swarts2025contractingstrategieselectrolyzerssecure}. These streams clarify
how applicants should be sequenced or accommodated after the available product
menu is specified. They do not generally make the strength of that menu
endogenous to a preceding evidence audit. Consequently, a planning model that
admits every proposed calendar may overstate queue relief, whereas a model that
rejects every imperfect calendar may discard potentially useful conditional access.

Temporal representation of the product creates a further distinction from
conventional planning models. Representative days, weeks, and clustering methods
are widely used to reduce computational burden in capacity-expansion and
energy-system models
\cite{poncelet_selecting_best_days,pineda2018_chronological}. Their limitations are particularly important when
storage, chronology, forecast errors, or extreme events affect investment
outcomes \cite{HILBERS2023120624,PINEDA20161113}. Here each selected season--hour-of-week cell becomes a
recurring and publishable restriction window: the canonical week is part of
the contractual object itself, and must be audited as product design, not only
as computational compression. The unresolved governance question is therefore
how evidence should determine the contractual strength counted as
queue-relieving capacity, without either accepting every proposed calendar or
foreclosing all non-firm access.

We address this gap through a one-way evidence-to-access sequence.
Stage A constructs and fixes directional calendars from residual-load and
price-timing signals. Stage B evaluates the submitted product, one common
bounded repair, and joint threshold attainability under the shared annual
budget; a separate aggregation-aware \PTDF\ branch distinguishes nominal from
robust structural support. Stage C treats the resulting ledger as fixed when
assigning horizon-stable routes, award years, operating recourse, and candidate
interventions. Submitted-calendar support can open $\NFK$, repair-only support
can open $\NFR$, and the uniformly derated fallback $\NFC$ is available only
through explicit policy authorization. Downstream allocation value cannot
retroactively validate the product, and none of the evidence objects is
interpreted as AC feasibility, contingency security, hosting capacity, or
operator approval. Although calibrated to the Netherlands, the architecture is intended for
settings where queue eligibility precedes a choice among firm, calendarized,
and adaptive access products, with tariff, compensation, and cost-allocation
details remaining jurisdiction-specific.

The contributions are fourfold:
\begin{itemize}[leftmargin=*,nosep]
\item a one-way construction--audit--allocation architecture that prevents
outcome-based validation of a candidate access product.
\item a solver-certified submitted-product and common-repair audit, a joint
oracle for threshold attainability under the shared annual budget, and an
aggregation-aware PTDF certificate.
\item a horizon-stable route interface that preserves applicant identity and
the prescribed ordering among selected applicants within each capacity pool.
\item complementary policy-package and target-free experiments separating
fallback authorization from evidence conditioning.
\end{itemize}

In the Dutch 2025--2035 testbed, 109 of 264 directional records support the
submitted or repaired calendar; the baseline joint threshold is attainable for
97 of 132 province--years and proven unattainable for 35. Eleven applicants
representing $934.5$~MW satisfy the full-horizon fixed-calendar gate, but none
is selected. In Experiment Set I, fallback authorization raises awards from
$2.119$ to $6.442$~GW and reduces backlog from $136.076$ to
$113.601$~GW-years. Across $H\in\{0,5,10\}$ and two queue rules, the routed
award gain remains $3.867$--$4.322$~GW. A cross-started target-free matched control finds no material objective or
returned-allocation effect of evidence conditioning once
the same fallback is available.

\section{Problem Formulation}
\label{section:problem_formulation}

We now formalize the three-stage mechanism summarized in
Fig.~\ref{fig:overall_framework_main}. The model first constructs a
calendarized candidate product, then audits the fixed product without using
downstream planning outcomes, and finally incorporates the resulting evidence
ledger into contract routing, queue allocation, and intervention planning.

\begin{figure}[t]
\centering
\begin{tikzpicture}[
>=Latex,
font=\scriptsize,
node distance=4.5mm,
stage/.style={
draw,
thick,
rounded corners=1.2mm,
align=center,
inner sep=4pt,
text width=0.88\columnwidth,
minimum height=10mm,
fill=white
},
arr/.style={->,thick}
]
\node[stage] (A) {
\textbf{Stage A: Candidate-calendar construction}\\[-0.5mm]
Public residual-load and price signals
$\longrightarrow$
fixed directional calendars
};
\node[stage, below=of A] (B) {
\textbf{Stage B: Product-evidence audit}\\[-0.5mm]
fixed calendars and required audit objects
$\longrightarrow$
directional evidence ledger
};
\node[stage, below=of B] (C) {
\textbf{Stage C: Contract routing and queue planning}\\[-0.5mm]
Fixed ledger, horizon-stable routes, applicants, and network options
$\longrightarrow$
contracts, awards, and interventions
};
\draw[arr] (A) --
node[right,font=\scriptsize] {fixed product} (B);
\draw[arr] (B) --
node[right,font=\scriptsize] {fixed ledger} (C);
\end{tikzpicture}
\caption{One-way evidence-to-access framework. Stage A constructs and fixes
the candidate product; Stage B audits that product and exports a directional
evidence ledger; Stage C treats the ledger as fixed when assigning contractual routes and determining access awards and interventions.}
\label{fig:overall_framework_main}
\end{figure}
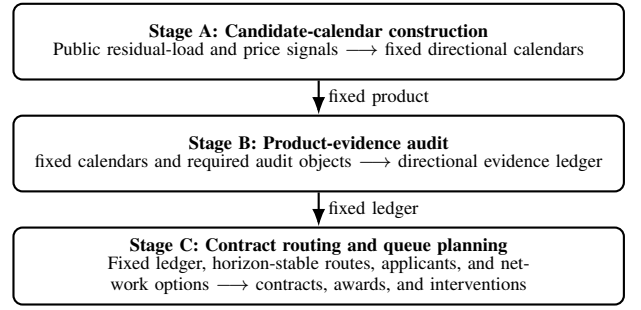

\subsection{Stage A: Candidate-Calendar Construction}
\label{subsec:stageA_calendar}

Stage A constructs a pair of directional restriction calendars for each
province and planning year from provincial residual load and national DA price
timing. GOPACS announcements, PTDF quantities, and Stage-C outputs are excluded
from construction.

\paragraph{Canonical policy domain}
For planning year $y$, let $\mathcal H_y$ denote the complete hourly UTC
calendar. Each physical hour $t\in\mathcal H_y$ is mapped to a season
$q(t)\in\mathcal Q$ and an hour of week $\phi_y(t)\in\mathcal K$, where
$\mathcal K:=\{0,\ldots,167\}$ and $\mathcal J:=\mathcal Q\times\mathcal K$.
The annual occupancy of cell $(q,k)\in\mathcal J$ is
\begin{equation}
o_{yqk}=\left|\left\{t\in\mathcal H_y:q(t)=q,\;\phi_y(t)=k\right\}\right|.
\label{eq:stageA_occupancy}
\end{equation}
For the equivalent canonical index $j=(q,k)$, we use the shorthand
\begin{equation}
\omega_{yj}:=o_{yqk}.
\label{eq:stageA_occupancy_shorthand}
\end{equation}
This shorthand is used in Stage B and in the accompanying technical report. Selecting a cell assigns the same directional restriction to every physical
hour represented by that season--hour-of-week pair. The canonical week is
therefore the policy domain of the product, not merely a computational device.

Provincial residual load is
\begin{equation}
RL^0_{nyt}=D_{nyt}-W^{\mathrm{on}}_{nyt}-W^{\mathrm{off}}_{nyt}-PV_{nyt},
\label{eq:stageA_residual_load}
\end{equation}
where $D_{nyt}$ is provincial demand and $W^{\mathrm{on}}_{nyt}$,
$W^{\mathrm{off}}_{nyt}$, $PV_{nyt}$ are onshore-wind, offshore-wind, and
photovoltaic generation. Its positive and negative parts define the two
directional pressures:
\begin{equation}
R^-_{nyt}=\left(RL^0_{nyt}\right)_+,\qquad
R^+_{nyt}=\left(-RL^0_{nyt}\right)_+.
\label{eq:stageA_directional_pressure}
\end{equation}
Positive residual load indicates conditions where additional withdrawal reinforces net demand, whereas negative residual load indicates conditions where additional injection accentuates modeled renewable surplus. Crucially, residual load is treated neither as a branch flow nor as evidence of local congestion.

Let $\operatorname{MM}_{ny}(\cdot)$ denote province--year min--max
normalization, returning the zero vector if the series is constant, and set
$r^\pm_{nyt}=\operatorname{MM}_{ny}(R^\pm_{ny\cdot})(t)$.

Let $p^{\mathrm{norm}}_{yt}\in[0,1]$ denote the normalized national
day-ahead price. This quantity is not interpreted as a province-specific
congestion price, a branch-loading measure, or evidence of physical
deliverability. It is included as a market-timing covariate. Residual-load
components are principal drivers of day-ahead electricity prices
\cite{Trebbien2023ElectricityPrices}, and national market prices can materially
affect the cost and usability of congestion-management contracts
\cite{VANDERHOLST2024110664}. A fixed-window product must therefore account
not only for the direction of local net-demand pressure but also for the market
exposure created by the timing of its restrictions.

To preserve this distinction, price modifies the score only when the
corresponding provincial residual-load direction is present. High-price hours
receive greater weight on the shortage side, whereas low-price hours receive
greater weight on the surplus side:
\begin{align}
S^-_{nyt}
&=
\operatorname{MM}_{ny}\!\left[
r^-_{ny\cdot}
\left(1+\eta_p p^{\mathrm{norm}}_{y\cdot}\right)
\right](t),
\nonumber\\
S^+_{nyt}
&=
\operatorname{MM}_{ny}\!\left[
r^+_{ny\cdot}
\left(1+\eta_p
\left(1-p^{\mathrm{norm}}_{y\cdot}\right)\right)
\right](t).
\label{eq:stageA_scores}
\end{align}
Thus price can reorder hours only within an already active local
residual-load direction. Setting $\eta_p=0$ gives the residual-load-only
control.

\paragraph{Candidate-calendar optimization}
Let $v^\pm_{nyqk}$ be the total score mass of the hours represented by cell
$(q,k)$. For each province--year pair, Stage A solves
\begin{equation}
\begin{aligned}
\max_{z^-,z^+}\quad
&\sum_{(q,k)\in\mathcal J}\!\left(v^-_{nyqk}z^-_{nyqk}+v^+_{nyqk}z^+_{nyqk}\right)\\
\mathrm{s.t.}\quad
&z^-_{nyqk}+z^+_{nyqk}\le1,\quad(q,k)\in\mathcal J,\\
&\sum_{(q,k)\in\mathcal J}o_{yqk}(z^-_{nyqk}+z^+_{nyqk})\le K_y,\\
&z^-_{nyqk},z^+_{nyqk}\in\{0,1\},\quad(q,k)\in\mathcal J,
\end{aligned}
\label{eq:stageA_calendar}
\end{equation}
where $K_y=\lceil\alpha|\mathcal H_y|\rceil$. Import and export share one
annual restriction-hour budget; no separate directional quota is imposed. The
headline $\alpha=0.15$ is an annual-duration analogue of an 85/15 availability
design, not a reconstruction of the operator's chronological TDTR activation
rule. The lifted calendars $w^0_{n,y,\dimp,t}=z^-_{n,y,q(t),\phi_y(t)}$
and $w^0_{n,y,\dexp,t}=z^+_{n,y,q(t),\phi_y(t)}$ are fixed before Stage B.

\subsection{Stage B: Product-Evidence Audit}
\label{subsec:stageB_audit}

Stage B evaluates each province--year--direction record after the Stage-A
calendar is fixed. Residual-load alignment measures directional pressure
timing, whereas price-timing alignment measures market exposure; neither is
a network-feasibility certificate. Because the core audit reuses construction
objects, it tests compliance with the declared timing requirements rather
than independently validating them. GOPACS and PTDF are separate diagnostics
and do not change the primary state. Throughout Stage B,
$\varepsilon=10^{-12}$ is a denominator safeguard in ratio calculations and
is absent from the repair and joint-oracle epigraph denominators.

The two core arrays are
\begin{align}
S^{\mathrm{RL},\dimp}_{nyt}
&=r^-_{nyt},
&
S^{\mathrm{RL},\dexp}_{nyt}
&=r^+_{nyt},
\nonumber\\
S^{\mathrm{PriceTiming},\dimp}_{nyt}
&=p^{\mathrm{norm}}_{yt},
&
S^{\mathrm{PriceTiming},\dexp}_{nyt}
&=1-p^{\mathrm{norm}}_{yt}.
\label{eq:stageB_core_arrays}
\end{align}
Price timing is a normative market-exposure requirement and is not
interpreted as a locational congestion measure.

For direction $d\in\{\dimp,\dexp\}$, define the residual-load mass
\begin{equation}
T^{\mathrm{RL},d}_{ny}
=
\sum_{t\in\mathcal H_y}S^{\mathrm{RL},d}_{nyt},
\qquad
\mu_{nyd}
=
\frac{T^{\mathrm{RL},d}_{ny}}
{T^{\mathrm{RL},\dimp}_{ny}
 +T^{\mathrm{RL},\dexp}_{ny}+\varepsilon}.
\label{eq:stageB_rl_materiality}
\end{equation}
The implemented directional-materiality flag is
\begin{align}
M_{nyd}
&=
\ind{\mu_{nyd}\ge\mu_{\min}}
\ind{T^{\mathrm{RL},d}_{ny}\ge10^{-9}},
\nonumber\\
\mu_{\min}
&=0.20.
\label{eq:stageB_materiality}
\end{align}
The $10^{-9}$ condition is a numerical nonzero guard on normalized evidence;
$\mu_{\min}$ supplies the substantive directional-materiality criterion.

The submitted calendar offers direction $d$ when
\begin{align}
\nu_{nyd}
&=
\frac{\sum_{t\in\mathcal H_y}w^0_{nydt}}
{K_y+\varepsilon},
\nonumber\\
O_{nyd}
&=
\ind{\nu_{nyd}\ge O_{\min}},
\qquad
O_{\min}=0.05.
\label{eq:stageB_offered_direction}
\end{align}

The active core-evidence set is
\begin{equation}
\mathcal M^{\mathrm{core}}_{nyd}
=
\left\{
m\in\{\mathrm{RL},\mathrm{PriceTiming}\}:
\begin{array}{l}
M_{nyd}=1,\\[-1mm]
V^{m,d,\star}_{K}(n,y)>0
\end{array}
\right\}.
\label{eq:stageB_active_core_set}
\end{equation}
Thus a core object enters the minimum only when it is present, active, and
material for the assessed direction.

For canonical cell $j\in\mathcal J$, aggregate physical-hour evidence as
\begin{equation}
s^{m,d}_{nyj}
=
\sum_{\substack{t\in\mathcal H_y:\\(q(t),\phi_y(t))=j}}
S^{m,d}_{nyt}.
\label{eq:stageB_cell_evidence}
\end{equation}
For an admissible directional calendar $w$ we set
\begin{align}
V^{m,d}_{ny}(w)
&=
\sum_{j\in\mathcal J}s^{m,d}_{nyj}w_j,
\nonumber\\
V^{m,d,\star}_{K}(n,y)
&=
\max_{z\in\{0,1\}^{|\mathcal J|}}
\left\{
\sum_{j\in\mathcal J}s^{m,d}_{nyj}z_j:
\sum_{j\in\mathcal J}\omega_{yj}z_j\le K_y
\right\}.
\label{eq:stageB_evidence_value}
\end{align}
When $\mathcal M^{\mathrm{core}}_{nyd}\ne\varnothing$, the submitted retained-value
score is
\begin{equation}
H_{0,nyd}
=
\min_{m\in\mathcal M^{\mathrm{core}}_{nyd}}
\frac{V^{m,d}_{ny}(w^0_{nyd})}
{V^{m,d,\star}_{K}(n,y)+\varepsilon}.
\label{eq:stageB_H0}
\end{equation}

Import and export share one annual restriction-hour currency, so repair is
computed once for the complete province--year product. Define
\begin{equation}
\mathcal I_{ny}
=
\left\{
(d,m):
d\in\{\dimp,\dexp\},\;
m\in\mathcal M^{\mathrm{core}}_{nyd}
\right\}.
\label{eq:stageB_joint_index}
\end{equation}
Let $\mathcal R_{B_{0y}}(w^0_{ny})$ contain joint binary calendars satisfying
directional non-overlap, the shared annual union budget, and
\begin{align}
\sum_{d\in\{\dimp,\dexp\}}
\sum_{j\in\mathcal J}
\omega_{yj}|w_{nydj}-w^0_{nydj}|
&\le B_{0y},
\nonumber\\
B_{0y}
&=
\left\lfloor0.02|\mathcal H_y|\right\rfloor .
\label{eq:stageB_repair_distance}
\end{align}
This is an occupancy-weighted directional-assignment distance. Switching one
canonical cell from import to export changes two binary assignments and
therefore contributes twice that cell's occupancy.

The bounded-repair problem is
\begin{equation}
\begin{aligned}
H^{R,\star}_{ny}
=
\max_{w,z}\quad &z\\
\mathrm{s.t.}\quad
&
V^{m,d}_{ny}(w_{nyd})
\ge
zV^{m,d,\star}_{K}(n,y),
\qquad (d,m)\in\mathcal I_{ny},\\
&
w\in\mathcal R_{B_{0y}}(w^0_{ny}),\\
&
0\le z\le1.
\end{aligned}
\label{eq:stageB_common_repair}
\end{equation}
When a feasible incumbent reaches $\rho_0$, denote the corresponding common
calendar by $w^R_{ny}$ and retain it as the repaired witness.
Its directional score is
\begin{equation}
H_{R,nyd}
=
\min_{m\in\mathcal M^{\mathrm{core}}_{nyd}}
\frac{V^{m,d}_{ny}(w^R_{nyd})}
{V^{m,d,\star}_{K}(n,y)+\varepsilon}.
\label{eq:stageB_HR}
\end{equation}
Repair certification uses solver incumbents and valid objective bounds. A feasible incumbent with objective at
least $\rho_0=0.80$ proves a pass; infeasibility or a valid upper bound below
$\rho_0$ proves a fail. Otherwise the solve remains unresolved. Unresolved
repairs are rerun and cannot enter the final ledger. 

The bounded repair tests whether the submitted product can be corrected within
a small edit neighborhood. To distinguish a larger redesign from a
shared-budget evidence conflict, we also solve the joint oracle without the edit-distance restriction
\begin{equation}
\begin{aligned}
H^{J,\star}_{ny}
=
\max_{w,z}\quad &z\\
\mathrm{s.t.}\quad
&
V^{m,d}_{ny}(w_{nyd})
\ge
zV^{m,d,\star}_{K}(n,y),
\quad(d,m)\in\mathcal I_{ny},
\\
&
w_{ny,\dimp,j}+w_{ny,\dexp,j}\le1,
\quad j\in\mathcal J,
\\
&
\sum_{j\in\mathcal J}\omega_{yj}
\left(
w_{ny,\dimp,j}+w_{ny,\dexp,j}
\right)
\le K_y,
\\
&
w_{nydj}\in\{0,1\},
\quad d\in\{\dimp,\dexp\},
\\
&
j\in\mathcal J,
\qquad
0\le z\le1.
\end{aligned}
\label{eq:stageB_joint_oracle}
\end{equation}
Relative to the bounded-repair problem,
\eqref{eq:stageB_joint_oracle} removes the edit-distance restriction while
retaining the same active evidence requirements, non-overlap rule, and
shared annual restriction-hour budget. A value
$H^{J,\star}_{ny}\ge\rho_0$ proves that at least one admissible common
calendar reaches the threshold. It does not by itself imply that the
submitted calendar failed or that redesign was necessary. Only when both
the submitted product and bounded repair fail does an oracle pass establish
that a redesign outside the repair neighborhood can satisfy the declared
requirements. A solver-certified upper bound below $\rho_0$ shows that the
active evidence requirements cannot be attained jointly within the declared
calendar family, non-overlap rule, and shared annual budget. We refer to
this outcome as a joint evidence conflict; it does not identify the
annual budget as the unique cause and does not imply physical infeasibility
or a need for reinforcement. The oracle is diagnostic and does not alter the
primary evidence state used by Stage C.

A same-budget common-hour construction is retained as a dynamic diagnostic.
It normalizes the active physical-hour evidence arrays, assigns each selected
hour to one direction, and selects at most $K_y$ non-overlapping hours.
Because this construction is a deterministic heuristic rather than an exact
adaptive-product optimization certificate, it does not assign a primary
evidence state and does not authorize an access route. Its complete definition
and numerical outputs are reported in the accompanying technical report.

The directional state belongs to
\[
\begin{gathered}
\{\ApproveFixed,\ApproveRepairable,\NoMaterialProduct,\\
\NoEvidence,\RejectReinforce\}.
\end{gathered}
\]

The implementation first assigns $\NoMaterialProduct$ when $M_{nyd}=0$ and
$\NoEvidence$ when the direction is material but no active core-evidence
object remains. A submitted joint-product budget or non-overlap violation
bypasses fixed support and is assigned $\ApproveRepairable$ only when the
common repair has a certified pass, reaches $0.80$, and satisfies the repaired
joint-product constraints. A material direction that is not offered is assigned
$\RejectReinforce$. Otherwise, the direction is fixed-supported when
$H_{0,nyd}\ge0.80$, repair-supported when a certified repair reaches
$H_{R,nyd}\ge0.80$, and fixed/repaired unsupported in every remaining case.
The dynamic construction, GOPACS warning, PTDF screen, and joint oracle are
separate diagnostics; none changes the primary evidence state. Stage B exports
the evidence states and corresponding calendar witnesses; it neither
authorizes MW capacity nor assigns the fallback route $\NFC$.

\paragraph{Aggregation-aware structural support}
Let $X_{nt}$ be province totals and reconstruct balanced bus injections as
$\widetilde x_t=WX_t+c_t$, where $W_{bn}\ge0$ is supported on buses assigned
to province $n$ and $\sum_b W_{bn}=1$. For PTDF matrix $P$, branch ratings
$\bar f_\ell>0$, and province bus set $\mathcal B_n$, define
\begin{align}
\bar P_{\ell n}&=\sum_{b\in\mathcal B_n}W_{bn}P_{\ell b},\quad
\delta_{\ell n}=\max_{b\in\mathcal B_n}|P_{\ell b}-\bar P_{\ell n}|,
\nonumber\\
\epsilon_t&=\max_{\ell\in\mathcal L}
\frac{\sum_n\delta_{\ell n}\widehat M_{nt}+|P_{\ell\cdot}c_t|}{\bar f_\ell},
\label{eq:mapping_error}
\end{align}
where $\widehat M_{nt}$ envelopes the unresolved within-province absolute
injection mass. The implemented choice $\widehat M_{nt}=2|X_{nt}|$ is a declared conditional
envelope; it is not a universal identity when substantial offsetting
within-province injections are unobserved. Let $\mathcal W^K_{nyd}$ denote the binary canonical calendars satisfying
$\sum_{j\in\mathcal J}\omega_{yj}w_j\le K_y$, with physical-hour lift
$w_t=w_{j_y(t)}$. For structural score $\widetilde G_t$, let
$\widetilde\rho_K(w^0)=V(w^0;\widetilde G)/V^\star_K(\widetilde G)$,
$E_0=\sum_t\epsilon_t w^0_t$, and
$E_K=\max_{w\in\mathcal W^K_{nyd}}\sum_t\epsilon_t w_t$.

\begin{theorem}[Aggregation-aware province-mapping certificate]
\label{thm:province_mapping_certificate}
Let $G_t,\widetilde G_t\ge0$ denote the structural loading scores from the
unknown balanced bus-level injections and the declared reconstruction,
respectively, with $w^0\in\mathcal W^K_{nyd}$. If the branch-stress
aggregator is one-Lipschitz, the declared mass envelope is valid, and the
oracle values are positive, then
\begin{equation}
\rho_K(w^0;G)\ge\left[\frac{\widetilde\rho_K(w^0)V^\star_K(\widetilde G)-E_0}
{V^\star_K(\widetilde G)+E_K}\right]_+,
\label{eq:province_mapping_bound}
\end{equation}
Hence the product is
aggregation-robust at threshold $\rho_{\mathrm{PTDF}}$ whenever the
right-hand side of \eqref{eq:province_mapping_bound} is at least that threshold.
\end{theorem}
\begin{IEEEproof}
Within-province PTDF dispersion and the mass envelope imply
$|G_t-\widetilde G_t|\le\epsilon_t$. Summing over the submitted calendar
shows that its value can decrease by at most $E_0$. Maximizing over the
admissible product domain shows that the oracle value can increase by at most
$E_K$. Combining these bounds yields
\eqref{eq:province_mapping_bound}; the complete derivation is given in the
accompanying technical report \cite{kryeziu2026report}.
\end{IEEEproof}

Failure of this sufficient certificate means only that aggregation-robust
support is not established under the declared reconstruction and mass
envelope. It does not establish physical infeasibility.

\subsection{Stage C: Contract Routing, Queue Allocation, and Intervention Planning}
\label{subsec:stage_c}

Stage C takes the directional ledger as fixed and determines applicant routes,
award years, backlog exposure, operating recourse, and candidate
interventions. The access menus are
\begin{align}
\mathcal X^{\mathrm{firm}}
&=\{\FA\},
\nonumber\\
\mathcal X^{\mathrm{un}}
&=\{\FA,\TD\},
\nonumber\\
\mathcal X^{\mathrm{sup}}
&=\{\FA,\NFK,\NFR\},
\nonumber\\
\mathcal X^{\mathrm{PTDF}}
&=\{\FA,\TD\},
\nonumber\\
\mathcal X^{\mathrm{rout}}
&=\{\FA,\NFK,\NFR,\NFC\}.
\label{eq:stageC_menus}
\end{align}
The generic $\TD$ benchmark uses the submitted Stage-A calendar and does not
use the evidence ledger. The PTDF-supported menu uses the same generic $\TD$
route but opens it only when all required horizon-wide nominal PTDF flags are
available and pass. Evidence-supported access instead opens $\NFK$ or $\NFR$
through the directional ledger. In the routed menu, $\NFK$ uses the submitted
calendar, $\NFR$ uses the effective repaired calendar, and $\NFC$ uses a
scalar planning deliverability factor.

For menu $r$, let
$\mathcal X^{r,F}=\mathcal X^r\cap\{\FA\}$ and
$\mathcal X^{r,N}=\mathcal X^r\setminus\{\FA\}$.
For $p\in\{F,N\}$, write
$\mathcal X^{r,p}=\mathcal X^{r,F}$ when $p=F$ and
$\mathcal X^{r,p}=\mathcal X^{r,N}$ when $p=N$.

Let $e_g\in\{0,1\}$ denote the common exogenous eligibility indicator used by
every non-firm menu. A missing, null, or nonbinary value is invalid and
excludes applicant $g$ from all non-firm routes; firm access is not conditioned
on $e_g$. Let
\begin{equation}
\mathcal Y_g^{\mathrm{op}}
=
\{y\in\mathcal Y:y\ge t_g\}
\label{eq:stageC_operating_horizon}
\end{equation}
denote the modeled operating years of applicant $g$, and define its relevant
direction set by
\begin{equation}
\mathcal D_g
=
\begin{cases}
\{d(g)\},&\text{if $g$ is unidirectional},\\
\{\dimp,\dexp\},&\text{if $g$ is bidirectional}.
\end{cases}
\label{eq:stageC_applicant_directions}
\end{equation}
The ledger must contain one state $c_{n(g),y,d}$ for every
$(y,d)\in\mathcal Y_g^{\mathrm{op}}\times\mathcal D_g$. If any required
record is unavailable, no strong route is certified and
$X_g^{\mathrm{sup}}=\varnothing$. Otherwise, the horizon-stable
evidence-supported route is
\begin{equation}
X_g^{\mathrm{sup}}
=
\begin{cases}
\NFK,
&
\substack{
c_{n(g),y,d}=\ApproveFixed\\
\text{for every }(y,d)\in
\mathcal Y_g^{\mathrm{op}}\times\mathcal D_g,
}
\\[2mm]
\NFR,
&
\substack{
c_{n(g),y,d}\in
\{\ApproveFixed,\ApproveRepairable\}\\
\text{for every }(y,d)\in
\mathcal Y_g^{\mathrm{op}}\times\mathcal D_g,\\
\text{and at least one required state is }\ApproveRepairable,
}
\\[2mm]
\varnothing,
&\text{otherwise}.
\end{cases}
\label{eq:stageC_supported_route}
\end{equation}
The routed-policy non-firm route is
\begin{equation}
X_g^{\mathrm{rout}}
=
\begin{cases}
\varnothing,
&e_g=0,\\
X_g^{\mathrm{sup}},
&e_g=1\text{ and }X_g^{\mathrm{sup}}\ne\varnothing,\\
\NFC,
&e_g=1\text{ and }X_g^{\mathrm{sup}}=\varnothing.
\end{cases}
\label{eq:stageC_routed_route}
\end{equation}
Thus missing evidence closes the strong $\NFK$ and $\NFR$ routes, but it does
not automatically close the explicitly authorized fallback. For an eligible
applicant, $\NFC$ is available only in a menu that authorizes it. The
admissible non-firm route is determined before the Stage-C optimization and is
constant across admissible award years.

Applicant $g$ requests $R_g$ MW, has energization year $t_g$, social-priority
indicator $s_g$, and baseline rank $\bar b_g$. Under priority horizon $H$,
the earliest consideration year is
\begin{equation}
p_g(H)
=
\max\{y_{\min},t_g-Hs_g\}.
\label{eq:stageC_priority_year}
\end{equation}
For the energization queue, applicants are ordered within province by
$(p_g(H),\bar b_g,\mathrm{ID}_g)$. The random benchmark uses a seeded
within-province permutation with seed 42. At $H>0$, the priority-adjusted year
remains the primary ordering key; at $H=0$, the implemented random benchmark
is the seeded within-province permutation.

Let $\delta_{gX},w_{gy}\in\{0,1\}$ and $a_{gyX}\ge0$ denote the route,
award-year, and awarded-capacity variables, respectively. Each accepted applicant receives one route and its full request in one
admissible year:
\begin{align}
\sum_{X\in\mathcal X^r}\delta_{gX}
&=
\sum_{y\in\mathcal Y}w_{gy}
\le1,
\nonumber\\
\sum_{X\in\mathcal X^r}a_{gyX}
&=
R_gw_{gy},
\nonumber\\
a_{gyX}
&\le R_g\delta_{gX},
\nonumber\\
a_{gyX}
&\le R_g\zeta^{rX}_{gy},
\qquad X\in\mathcal X^{r,N},
\nonumber\\
w_{gy}
&=0,
\qquad y\notin[p_g(H),t_g].
\label{eq:stageC_awards}
\end{align}
For each non-firm menu, $\zeta^{rX}_{gy}=1$ only when $e_g=1$ and route $X$
satisfies the corresponding menu condition above; otherwise
$\zeta^{rX}_{gy}=0$.
Define $A^X_{gy}=\sum_{\tau\le y}a_{g\tau X}$. Firm and non-firm queues are
enforced separately. If $\Delta_g^p$ indicates selection of pool $p\in\{F,N\}$
and $w^p_{gy}$ linearizes $w_{gy}\Delta_g^p$, every earlier--later pair
$i\prec_H j$ satisfies
\begin{equation}
\sum_{\tau\le y}w^p_{j\tau}\le\sum_{\tau\le y}w^p_{i\tau}+1-\Delta_i^p.
\label{eq:stageC_precedence}
\end{equation}
Since $\Delta_i^p=0$ relaxes the constraint for an unselected predecessor,
\eqref{eq:stageC_precedence} orders selected applicants within a pool rather
than imposing strict first-come-first-served admission.

With intervention availability $\chi_{ay}$ and province capacity credit
$\Delta^p_{an}$,
\begin{align}
K^p_{ny}(x)
&=
\bar K^p_{ny}
+
\sum_a\chi_{ay}\Delta^p_{an}x_a,
\nonumber\\
\sum_{g:n(g)=n}
\sum_{X\in\mathcal X^{r,p}}A^X_{gy}
&\le
K^p_{ny}(x),
\qquad p\in\{F,N\}.
\label{eq:stageC_pools}
\end{align}
All non-firm routes consume the common non-firm pool. Routing changes the
contractual representation and operating exposure of an award; it does not
create route-specific physical capacity.

Let $z^\pm_{nyqk}$ be the submitted calendar and $z^{R,\pm}_{nyqk}$ the
effective certified calendar. Route deliverability is
\begin{align}
\alpha^{\dimp}_{nyqk,\FA}
=\alpha^{\dexp}_{nyqk,\FA}
&=1,
\nonumber\\
(\alpha^{\dimp}_{nyqk,\TD},\alpha^{\dexp}_{nyqk,\TD})
&=(1-z^-_{nyqk},1-z^+_{nyqk}),
\nonumber\\
(\alpha^{\dimp}_{nyqk,\NFK},\alpha^{\dexp}_{nyqk,\NFK})
&=(1-z^-_{nyqk},1-z^+_{nyqk}),
\nonumber\\
(\alpha^{\dimp}_{nyqk,\NFR},\alpha^{\dexp}_{nyqk,\NFR})
&=(1-z^{R,-}_{nyqk},1-z^{R,+}_{nyqk}),
\nonumber\\
\alpha^d_{nyqk,\NFC}
&=\delta_C,
\qquad d\in\{\dimp,\dexp\}.
\label{eq:stageC_deliverability}
\end{align}
The uniformly derated fallback coefficient $\delta_C$ is a planning
deliverability factor, not an interruption probability or a substitute for
security simulation.

\paragraph{Operating recourse and objective}
Deliverable capacity per direction is
\begin{equation}
D^d_{nyqk\sigma}
=
\sum_{X\in\mathcal X^r}
\alpha^d_{nyqk,X}C_{ny\sigma X},
\qquad
d\in\{\dimp,\dexp\},
\label{eq:stageC_deliverable_capacity}
\end{equation}
where $C_{ny\sigma X}$ aggregates active operational capacity by province,
year, deterministic signature, and route. For a static signature, the
corresponding served quantity satisfies
\begin{equation}
0\le u^d_{nyqk\sigma}
\le
\varphi_{nyqk\sigma}D^d_{nyqk\sigma},
\qquad
d\in\{\dimp,\dexp\}.
\label{eq:stageC_recourse}
\end{equation}
Generation signatures use $d=\dexp$, whereas load signatures use
$d=\dimp$. The complete generation, load, and storage recourse formulation is
given in the accompanying technical report.
Storage uses signature-level charge, discharge, and cyclic state-of-charge
variables with route-weighted import/export limits. The annual
fixed service charge $F_{ny}$ is applied once per active applicant, the
contracted-capacity charge $\tau^{\mathrm{MWc}}_{ny}$ is applied only to
active firm capacity, and the volumetric transport-energy charge $\tau^E_{ny}$
enters operating recourse. Thus
\begin{equation}
\Psi^{\mathrm{fix},r}
=\sum_{g,y}\varrho_y\left[
F_{n(g)y}E_{gy}
+\tau^{\mathrm{MWc}}_{n(g)y}C^{\mathrm{op},\FA}_{gy}
\right].
\label{eq:stageC_fixed_bill}
\end{equation}
Monthly measured-peak coefficients are retained only for provenance and do not
enter the objective. These dominant-DSO coefficients are reduced-form planning
proxies, not the regulated TenneT TDTR tariff. No awarded MW is charged both
the firm contracted-capacity term and the fallback adjustment.

The repaired-route premium is charged once on an $\NFR$ award:
\begin{equation}
\Pi^{R,r}
=\lambda_R^r
\sum_{g,y}\varrho_y\tau^{\mathrm{MWc}}_{n(g)y}a_{gy,\NFR},
\label{eq:stageC_repair_premium}
\end{equation}
where $\lambda_R^r=0.25$ for the routed menu in Experiment Set I and is zero
in Experiment Set II. No final solve awards $\NFR$, so this term is zero in
the reported outcomes.

The fallback cost adjustment is
\begin{equation}
\Pi^{C,r}
=(s_C^{\mathrm{cost}}-1)
\sum_{g,y}\varrho_y\tau^{\mathrm{MWc}}_{n(g)y}
C^{\mathrm{op},\NFC}_{gy}.
\label{eq:stageC_nfc_term}
\end{equation}
It is an explicit model stress adjustment benchmarked to the
contracted-capacity coefficient, not a second regulated tariff; it vanishes
when $s_C^{\mathrm{cost}}=1$.

\paragraph{Experiment Set I: primary policy-package evaluation}
Set I uses $(\delta_C,s_C^{\mathrm{cost}})=(0.85,1.00)$ and minimizes
\begin{align}
J_{\mathrm{PP}}^r
={}&\mathrm{CAPEX}+\Psi^{\mathrm{fix},r}
+\gamma_{\mathrm{rec}}\Phi^r
+\Pi^{R,r}+\Pi^{C,r}
-\gamma_{\mathrm{val}}V^r
\nonumber\\
&-\lambda_{\mathrm{conn}}C^{\mathrm{tot},r}
+\Pi^{\mathrm{slack},r}.
\label{eq:stageC_policy_objective}
\end{align}
The terms are, respectively, investment, fixed network charges, operating
recourse, route adjustments, declared project value, a terminal connection
score, and penalties for stress-, connection-, and non-firm-utilization target
shortfalls. Set I also prohibits node $n$ from exporting in its own submitted
shortage-window cells. The objective is a composite planning scalarization, not
a welfare measure.

Backlog is counted from the applicant's energization year:
\begin{align}
B^r
&=\sum_{g,y}\ind{y\ge t_g}
\left(R_g-\sum_XA^X_{gy}\right),
\nonumber\\
\operatorname{BRI}^r
&=\frac{B^{\mathrm{firm}}-B^r}
{B^{\mathrm{firm}}-B^{\mathrm{un}}}.
\label{eq:stageC_backlog}
\end{align}
BRI may exceed one because indivisible requests can change the accepted
portfolio. For a fixed accepted set, backlog is independent of the particular
pre-energization award year.

\paragraph{Experiment Set II: target-free mechanism decomposition}
Set II disables the direct coverage and utilization targets, terminal score,
and node-specific export restriction while retaining the remaining physical,
economic, queue, and intervention inputs. It minimizes
\begin{equation}
J_{\mathrm{TF}}^r
=
\mathrm{CAPEX}
+\Psi^{\mathrm{fix},r}
+\gamma_{\mathrm{rec}}\Phi^r
+\Pi^{R,r}
+\Pi^{C,r}
-\gamma_{\mathrm{val}}V^r.
\label{eq:stageC_target_free_objective}
\end{equation}
``Target-free'' therefore means free of the explicit uptake targets used in
Set I, not assumption-free, policy-neutral, or welfare-maximizing. Within this
set, EvidenceSupported versus EvidenceRouted identifies fallback authorization.
The state-blind menu
\begin{equation}
\mathcal X^{\mathrm{blind}}=\{\FA,\NFC\}
\label{eq:stageC_blind_menu}
\end{equation}
opens $\NFC$ to every otherwise non-firm-eligible applicant independently of
the ledger. EvidenceRouted and BlindFallback otherwise share every applicant,
queue, capacity, intervention, fallback, operating, and objective input.

The headline Set-II calibration is $(0.70,1.50)$. The complete target
definitions, objective coefficients, robustness calibrations, operating
constraints, and solver settings are reported in the accompanying technical
report.
\section{Numerical Results}
\label{section:numerical_results}

\paragraph{Data and scenario construction}
The testbed covers the twelve Dutch provinces over 2025--2035. Quality-
controlled national hourly load is allocated to provincial sector profiles
using CBS annual shares; historical solar and wind shapes are scaled to
declared annual planning trajectories, with offshore production allocated
through year-specific landing shares. A leap-aware national day-ahead price
model supplies the market-timing signal. The untouched 2025 GOPACS record is
used only as an external holdout, while the PTDF branch uses a declared public
province-to-bus reconstruction. The $9.1$ GW baseline non-firm pool is an
exogenous planning envelope rather than published province-specific hosting
capacity. Full provenance and parameters are reported in the
technical report \cite{kryeziu2026report}.

\begin{table}[t]
\centering
\scriptsize
\setlength{\tabcolsep}{2.1pt}
\renewcommand{\arraystretch}{1.02}
\caption{Fixed parameters in the headline main-paper cases. Construction grids,
robustness values, target definitions, and solver settings are reported in the
accompanying technical report.}
\label{tab:fixed_parameters}
\begin{tabularx}{\columnwidth}{@{}p{0.18\columnwidth}
p{0.43\columnwidth}X@{}}
\toprule
Block & Fixed value & Role\\
\midrule
Stage A
& $D1,\ \alpha=0.15,\ \eta_p=1$
& Calendar construction.\\
Stage B
& $\rho_0=0.80,\ \mu_{\min}=0.20,\ O_{\min}=0.05$
& Evidence thresholds.\\
Repair/PTDF
& $B_{0y}=\lfloor0.02|\mathcal H_y|\rfloor,\ 
\rho_{\mathrm{PTDF}}=0.80$
& Repair radius and structural-screen threshold.\\
Stage C
& $H=5$, energization queue
& Primary priority specification.\\
Set I
& $(\delta_C,s_C^{\mathrm{cost}})=(0.85,1.00)$
& Policy-package fallback.\\
Set II
& $(\delta_C,s_C^{\mathrm{cost}})=(0.70,1.50)$
& Matched-control fallback.\\
Weights
& $\gamma_{\mathrm{rec}}=\gamma_{\mathrm{val}}=1$;
$\lambda_{\mathrm{conn}}=1$ in Set I and $0$ in Set II
& Objective scaling.\\
\bottomrule
\end{tabularx}
\end{table}

\paragraph{Locked candidate}
The headline candidate is the regular $D1$ design and is frozen before the
GOPACS holdout, PTDF validation, Stage-B audit, uncertainty analysis, and
Stage-C allocation. The configured construction grid contains the
residual-load-only control, twelve $\alpha$--$\eta_p$ combinations, two
forecast-error diagnostics, and two price-ramp diagnostics; the complete grid
is reported in the accompanying technical report. Every executed run satisfies
the shared annual union budget and directional non-overlap. In the locked
calendar, $78.4\%$ of selected hours are import-side and $21.6\%$ are
export-side. 

\begin{figure}[t]
\centering
\includegraphics[width=\linewidth]
{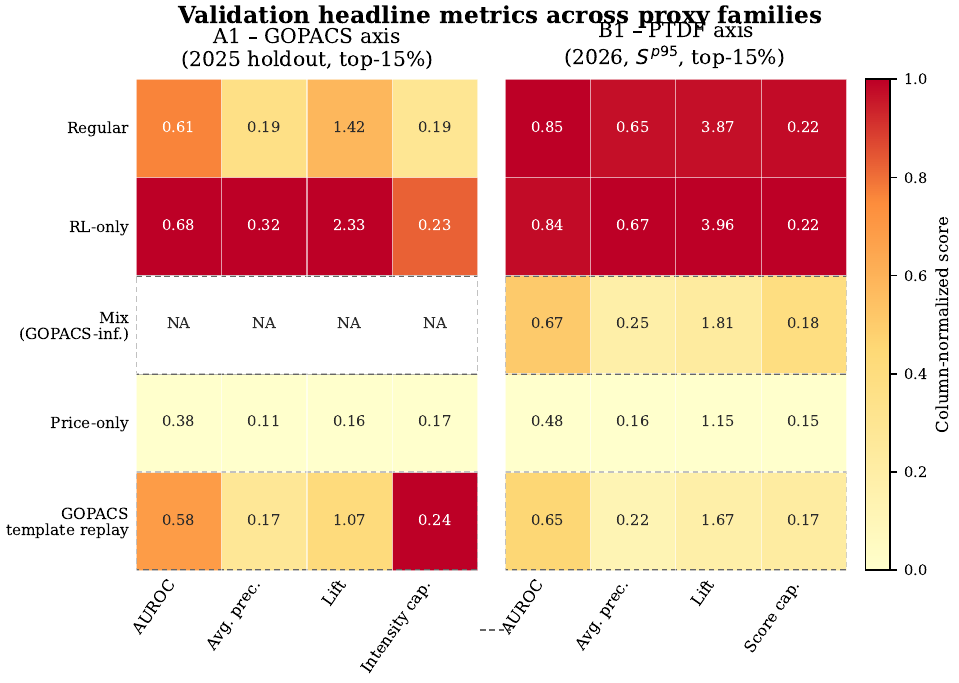}
\caption{Role-separated validation of Stage-A proxy families: (a) alignment
with observed GOPACS activity in the 2025 holdout and (b) alignment with the
PTDF $p95$ structural-loading target in 2026. Neither object is used to
construct or select the headline product.}
\label{fig:validation}
\end{figure}

Panel (a) shows that residual load carries most GOPACS-ranking information: its
AUROC, average precision, and top-$15\%$ lift are $0.68$, $0.32$, and $2.33$,
compared with $0.61$, $0.19$, and $1.42$ for D1. In Panel (b), the proxies are
nearly identical, with AUROC $0.84$--$0.85$ and lift $3.87$--$3.96$. Price
timing is therefore a distinct market-exposure requirement rather than a local
congestion predictor; it is the limiting retained-value component in all 109
fixed- or repair-supported directional records. In 100 joint demand, wind,
solar, and price perturbations, median coverage at severity threshold $0.75$
falls by $3.81$ percentage points for import and $9.49$ points for export;
this is a fixed-window stress test, not a stochastic-planning guarantee.

\subsection{Directional Evidence Ledger}
\label{subsec:evidence_ledger}
 
\begin{table}[t]
\centering
\footnotesize
\setlength{\tabcolsep}{3.6pt}
\renewcommand{\arraystretch}{1.04}
\caption{Final Stage-B directional evidence-support ledger.}
\label{tab:ledger}
\begin{tabular}{lrrr}
\toprule
Evidence state & Import & Export & Total \\
\midrule
Fixed-supported              & 93 & 15 & 108 \\
Repair-supported             &  0 &  1 &   1 \\
Low directional materiality  &  8 & 89 &  97 \\
Core evidence unavailable    &  0 &  0 &   0 \\
Fixed/repaired unsupported   & 31 & 27 &  58 \\
\midrule
Total                         &132 &132 & 264 \\
\bottomrule
\end{tabular}
\end{table}

Overall, 108 records support the submitted fixed calendar and one supports the
common bounded repair. Support is strongly directional: 93 province--year
pairs support only import, 16 support only export, and 23 support neither
direction as a fixed or repaired product; none supports both directions. The
sole repair-supported record is Zeeland export in 2033, where 156
occupancy-weighted directional hour edits raise the retained score from
$0.7988$ to $0.8337$. The common repair is solved once per province--year:
36 unique joint repair problems are attempted, one has a threshold-passing
certificate, 35 have solver-certified failure certificates, and none remains unresolved.

The joint oracle is resolved for all 132 province--years. The baseline
threshold is attainable for 97 products and proven unattainable for 35. Only
one submitted product fails while the joint oracle passes, and that case is
already resolved by the bounded repair; no baseline product fails both the
submitted and bounded-repair tests while passing the joint oracle. The oracle's
empirical role at the baseline is therefore to certify the 35 shared-budget
conflicts rather than to reveal a separate extensive-redesign class.

\begin{figure*}[t]
\centering
\includegraphics[width=0.88\textwidth]
{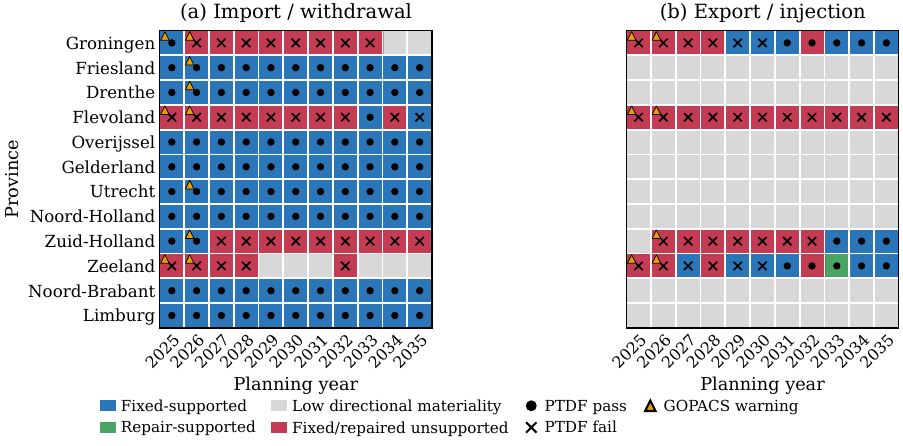}
\caption{Directional Stage-B evidence ledger, 2025--2035. Cell color reports
the annual evidence state for import/withdrawal (left) and export/injection
(right). PTDF symbols appear only where the nominal structural screen is
available, and triangles denote the GOPACS audit warning. An unmarked cell
therefore does not indicate PTDF failure. These diagnostics are evidence
screens rather than AC-feasibility, hosting-capacity, or connection-approval
results.}
\label{fig:directional}
\end{figure*}

Figure~\ref{fig:directional} reveals a pronounced directional and spatial
structure in the ledger, but not a universal coastal--inland divide. Drenthe,
Friesland, Gelderland, Limburg, Noord-Brabant, Noord-Holland, Overijssel, and
Utrecht form a stable load-dominated group: all $88$ import-side
province--year records are fixed-supported, whereas all $88$ corresponding
export-side records have low directional materiality. Groningen, Flevoland,
Zuid-Holland, and Zeeland contain all interannual evidence-state transitions
and the material export-side cases. Their import and export evidence evolves
differently over the planning horizon, so a single long-lived fixed calendar
is more difficult to justify from the available public evidence. Friesland and
Noord-Holland are coastal yet follow the stable load-dominated profile.
Accordingly, the relevant screening variables are the \emph{direction and temporal stability of the normalized net-injection evidence}, rather than a
binary coastal--inland label. For prospective developers this distinction is
material: electrolyzers, data centers, e-boilers, and other large loads should
primarily examine the import-side trajectory; generators should examine the
export-side trajectory; and storage projects must assess both because charging
and discharging use opposite access directions. Where the evidence state
changes repeatedly across years, shorter validity periods, scheduled
re-audits, or a conservative conditional route are more defensible than a
strong multi-year fixed-calendar promise.

The nominal PTDF screen provides a separate structural diagnostic. It retains
92 of 93 approved import records and 11 of 16 approved or repaired export
records. No record passes the sufficient aggregation-aware certificate under
$\widehat M_{nt}=2|X_{nt}|$. This result means only that the available
province-level data do not establish aggregation-robust support under the
declared envelope; it does not imply physical infeasibility or failure of an
operator connection study.

Annual directional support also does not automatically produce a selected
strong contract. Under the horizon-stable rule, 14 applicants representing
$977.5$~MW are $\NFK$-eligible over one operating year; the count declines to
12 over three years and to 11 over five years and over the full modeled
horizon. The 11 full-horizon applicants represent $934.5$~MW, while no
applicant is $\NFR$-eligible. None of these strong-route-eligible applicants is
selected by the final queue-allocation model. EvidenceSupported and
PTDFSupported therefore coincide with FirmOnly because of the final allocation
outcome, not because strong-route eligibility is absent. Moreover, across the
declared retained-value, materiality, and residual-load-only diagnostics, the
full-horizon result remains 11 $\NFK$-eligible applicants and zero
$\NFR$-eligible applicants.

\subsection{Experiment Set I: Policy-Package Evaluation}
\label{subsec:access_consequences}

\begin{table}[t]
\centering
\footnotesize
\setlength{\tabcolsep}{1.9pt}
\renewcommand{\arraystretch}{1.04}
\caption{Stage-C outcomes in Experiment Set I ($H=5$, energization queue).
Displayed values are rounded; differences in the text use unrounded outputs.}
\label{tab:headline}
\begin{tabular}{lrrrr}
\toprule
Menu & Award & Non-firm & Backlog & BRI\\
&[GW]&[GW]&[GW-yr]&\\
\midrule
Firm only          &2.119&0.000&136.076&0.000\\
Unscreened calendar&6.482&4.727&113.821&1.000\\
Evidence-supported &2.119&0.000&136.076&0.000\\
PTDF-supported     &2.119&0.000&136.076&0.000\\
Evidence-routed    &6.442&4.489&113.601&1.010\\
\bottomrule
\end{tabular}
\end{table}

EvidenceSupported and PTDFSupported coincide with FirmOnly, each awarding
$2.119$ GW and producing $136.076$ GW-years of backlog. EvidenceRouted awards
$6.442$ GW, including $4.489$ GW through $\NFC$, and reduces backlog to
$113.601$ GW-years. No capacity is assigned through $\NFK$ or $\NFR$.
Using unrounded outputs, fallback authorization adds $4.322$ GW and reduces
backlog by $22.475$ GW-years relative to EvidenceSupported. The routed award
is $40$ MW below the unscreened award and produces $0.220$ GW-years less
backlog, yielding $\operatorname{BRI}=1.010$. Set I therefore measures the
planning consequence of fallback authorization within the complete policy
package, not take-up of a strong calendar-specific route.

\subsection{Experiment Set II: Target-Free Mechanism Decomposition}
\label{subsec:blind_fallback_control}
 
\begin{table}[t]
\centering
\footnotesize
\setlength{\tabcolsep}{3.6pt}
\renewcommand{\arraystretch}{1.04}
\caption{Experiment-Set-II outcomes and cross-started matched control
($H=5$, energization queue, policy calibration).}
\label{tab:blind_fallback_control}
\begin{tabular}{lrrr}
\toprule
Menu & Award & Backlog & $\Delta J_{\mathrm{TF}}$ vs.\ blind\\
& [GW] & [GW-yr] & [EUR]\\
\midrule
Evidence-supported & 2.022 & 136.955 & --\\
Evidence-routed    & 6.442 & 113.601 & $-0.006$\\
Blind fallback     & 6.442 & 113.601 & $0.000$\\
\bottomrule
\end{tabular}
\end{table}
EvidenceRouted awards $4.420$~GW more and reduces backlog by
$23.354$~GW-years relative to EvidenceSupported. This contrast identifies the
effect of authorizing $\NFC$ when no selected horizon-stable fixed- or
repaired-calendar award is available. To isolate evidence conditioning, the
policy-calibration EvidenceRouted and BlindFallback cases are rerun with
opposite-case MIP starts under the same model inputs. Their primary objectives
differ by $-0.005859$ EUR, well inside the declared EUR~1 equivalence
tolerance. They return the same 31 accepted applicants, award years,
routes, $6.4416$~GW of total awards, $4.4521$~GW on $\NFC$, and the same
14-intervention portfolio. Hence no material primary-objective or returned
allocation effect of evidence conditioning is identified once the same
fallback is available. The auxiliary lexicographic tie-break did not certify
route-assignment uniqueness, so uniqueness among all primary-optimal solutions
is not claimed.

All thirteen core Stage-C cases select the same 14-intervention portfolio:
EUR $2.998$ billion of modeled CAPEX, $4.300$ GW of firm headroom, and
$0.560$ GW of non-firm-compatible headroom. All numerical and feasibility
gates pass, so menu differences arise from fallback authorization and applicant
assignment rather than different interventions.

\subsection{Interpretation and Scope Conditions}
\label{subsec:interpretation_scope}

The framework separates construction independence, external proxy validation,
product evidence, fallback authorization, and evidence conditioning. The
Stage-B retained-value benchmark is deliberately demanding because each
direction is compared with a full-budget oracle while import and export share
one annual budget. Materiality and offered-direction checks are therefore
reported separately. Bounded repair asks whether local edits suffice; 
the joint oracle asks whether any admissible common calendar can meet the
threshold. A solver-certified failure establishes incompatibility of the
active evidence requirements within the declared product constraints; it does
not establish physical-network infeasibility.

The PTDF results have the same bounded interpretation. The nominal score
evaluates one declared province-to-bus reconstruction, whereas the
aggregation-aware certificate is conditional on the stated within-province
mass envelope. Failure of that sufficient certificate means that the stronger
aggregation-robust claim is not established under the declared reconstruction
and uncertainty envelope. So, it is not evidence that an operator study
would reject the connection.

Experiment Set I retains the complete policy package and measures the
consequence of authorizing the routed fallback relative to strict evidence
support. Experiment Set II removes the direct uptake targets and introduces
the state-blind control. Within Set II, EvidenceSupported versus
EvidenceRouted identifies fallback authorization, whereas EvidenceRouted
versus BlindFallback isolates evidence conditioning conditional on fallback
availability. Objective and BRI values are defined within an experiment set
and should not be compared across sets as a one-parameter sensitivity.
 
\subsection{Policy Implications}
\label{subsec:policy_implications}

The evidence ledger is best interpreted as a pre-application screening and
contractual due-diligence instrument. It does not replace network-specific AC,
voltage, contingency, redispatch, protection, or implementation studies.

The two experiment sets distinguish hard closure, evidence-routed fallback,
and state-blind fallback. In the strict Experiment-Set-II matched control, the
latter two policies attain equivalent primary objectives and return the same
accepted applicants, award years, routes, and intervention decisions.
The demonstrated incremental role of the ledger is therefore the evidentiary
basis and contractual diagnosis of the offer, not the creation of additional
non-firm capacity once fallback access is already available.

For implementation, an offer should disclose the evidence state and data vintage,
restriction rule and annual budget, fallback deliverability, review dates,
notification and compensation provisions, downgrade conditions, and the
requirements for operator study, reinforcement, and conversion to firm access.

\section{Conclusion}
\label{sec:conclusion}

This paper develops a one-way evidence-to-access framework for calendarized
non-firm connections. Stage A fixes directional candidate calendars;
Stage B applies a solver-certified submitted-product and bounded-repair audit,
supplemented by a joint oracle and an aggregation-aware PTDF
certificate; and Stage C assigns horizon-stable routes before queue and
intervention decisions.

Of 264 directional records, 109 support the submitted or repaired calendar.
The joint threshold is attainable for 97 of 132 province--years and proven
unattainable for 35, yet no applicant receives an $\NFK$ or $\NFR$ award.
In Experiment Set I, authorizing $\NFC$ raises awards from $2.119$ to
$6.442$ GW and reduces backlog from $136.076$ to
$113.601$ GW-years; all $4.489$ GW of non-firm awards use the fallback.

The cross-started Experiment-Set-II matched control confirms that fallback
availability accounts for the volume effect: EvidenceRouted and BlindFallback
attain primary objectives equivalent within EUR~1 and return the same accepted
applicants, award years, routes, and intervention portfolio. Route-mix
uniqueness among all primary-optimal solutions is not claimed. The ledger's
demonstrated role is therefore contractual and diagnostic: it determines the
evidentiary strength and diagnosis of the offered right rather than creating
additional capacity once the same fallback is available.

\bibliographystyle{IEEEtran}
\bibliography{main}

\end{document}